\documentclass[hha]{article}
\usepackage{amsmath}
\usepackage{amssymb}
\usepackage{amsthm}
\usepackage[mathscr]{eucal}
\usepackage[all,knot,cmtip]{xy}
\xyoption{arc}

\numberwithin{equation}{section}
\newtheorem{definition}{Definition}[section]

\newtheorem{thm}[definition]{Theorem}

\newtheorem{ex}[definition]{Example}

\begin{document}
\title{On a simplicial monoid whose underlying simplicial set is not a quasi-category}
\author{Ryo Horiuchi}
\date{}

\maketitle

\section{Introduction}
It is well known that the underlying simplicial set of any simplicial group is a Kan complex (also known as $\infty$-groupoid or $(\infty, 0)$-category). Roughly speaking, Kan complex is an infinite-dimensional analogue of groupoid, and the relation between groupoids and categories resembles that between groups and monoids. Thus one may ask if the underlying simplicial set of each simplicial monoid is a quasi-category  (also known as  $\infty$-category or $(\infty, 1)$-category or weak Kan complex). In this short note, we construct a simplicial monoid whose underlying simplicial set is not a quasi-category.

Again it is well known that for any commutative group $G$ and any $n\in\mathbb{N}$ one can construct a simplicial group $\operatorname{K}(G, n)$ called the $n$-th Eilenberg-MacLane space of $G$, which is thus a Kan complex. This assignment gives rise to a functor from commutative groups to spectra, which again gives rise to a functor from rings to ring spectra that is crucial for the theory of higher algebra.

We apply this Eilenberg-MacLane space construction to commutative monoids to construct a simplicial monoid $\operatorname{K}(\mathbb{N}, 2)$ and show this does not satisfy a certain condition that any quasi-category must do. 

\section{Theorem}
Eilenberg-MacLane space construction for commutative monoids is already established. For example, see 3.1 in \cite{DLORV}. We simply apply such construction to commutative monoids instead of commutative groups. Let us first briefly recall the construction and see an example.

For a commutative monoid $M$ and a natural number $n\geq 1$, we have a simplicial monoid $\operatorname{K}(M, n)$ whose $k$-simplices is the reduced free monoid $M(S^n[k])$ generated by the set $S^n[k]$ of $k$-simplices in $S^n$ over $M$, where $S^n$ denotes the simplicial set $\Delta[n]/\partial\Delta[n]$. The simplicial structure on $\operatorname{K}(M, n)$ is induced from that on $S^n$. We may write $\operatorname{K}(M, 0)$ for the discrete simplicial monoid associated to $M$.

\begin{ex}{\rm Let $\mathbb{N}$ be the commutative monoid of natural numbers under the usual addition and $0$.
We write $\alpha(0)\alpha(1)...\alpha(m)$ for any map $\alpha:[m]\to[n]$ in $\Delta$. Then the sets of low dimensional simplicies of $S^2=\Delta[2]/\partial\Delta[2]$ can be given as the following:
\[S^2[0]=S^2[1]=\{*\}, \ 
S^2[2]=\{*, 012\}, \ 
S^2[3]=\{*, 0012, 0112, 0122\}.
\]
Therefore $\operatorname{K}(\mathbb{N}, 2)[0]=\operatorname{K}(\mathbb{N}, 2)[1]=*$ is the trivial monoid,
\[\operatorname{K}(\mathbb{N}, 2)[2]=\mathbb{N}\langle012\rangle\]
and
\[\operatorname{K}(\mathbb{N}, 2)[3]=\mathbb{N}\langle0012\rangle\oplus\mathbb{N}\langle0112\rangle\oplus\mathbb{N}\langle0122\rangle,\]
where, $\mathbb{N}\langle x\rangle$ denotes the free monoid generated by $x$. Moreover face maps $d_i:\operatorname{K}(\mathbb{N}, 2)[3]\to\operatorname{K}(\mathbb{N}, 2)[2]$ for $i\in\{0, 1, 2, 3\}$ are given by
\[d_0(a,b,c)=a, \ d_1(a,b,c)=a+b, \ d_2(a,b,c)=b+c, \ d_3(a,b,c)=c\]
for any $(a, b, c)\in\mathbb{N}\langle0012\rangle\oplus\mathbb{N}\langle0112\rangle\oplus\mathbb{N}\langle0122\rangle$}

\end{ex}

For every $[n]\in\Delta$ and $k\in[n]$, we let $\Lambda^k[n]$ denote the subsimplicial set of $\Delta[n]$ generated by its faces $\{\delta_i\in\Delta[n] \ | \ i\in[n]\setminus\{k\}\}$.

\begin{definition}[\cite{Joyal}]A simplicial set X is a quasi-category if it satisfies the following lifting property:

for any map \[\Lambda^k[n]\to X\] with $[n]\in\Delta$ and $k\in[n]\setminus\{0, n\}$, there exists a lift \[\Delta[n]\to X.\]
\end{definition}

\begin{thm}The Eilenberg-MacLane simplicial set $\operatorname{K}(M, n)$ for a commutative monoid $M$ is not necessarily a quasi-category. 
\end{thm}

\begin{proof}Take $M$ to be $\mathbb{N}$ and $n$ to be 2. Consider a map
\[\Lambda^1[3]\to\operatorname{K}(\mathbb{N}, 2)\]
such that the 0-face maps to an arbitrary natural number, the 2-face to 1 and the 3-face to 3. Assume there is a lift
\[\Delta[3]\to \operatorname{K}(\mathbb{N}, 2)\]
and denote the image of the unique non-degenerate 3-simplex by $(a, b, c)\in\mathbb{N}\langle0012\rangle\oplus\mathbb{N}\langle0112\rangle\oplus\mathbb{N}\langle0122\rangle$.
As we have observed in the example above, its 0-face would be $a$, 2-face $b+c$, and 3-face $c$. Thus we should have $b+3=1$ with $b\in\mathbb{N}$. Therefore such lift can not exist. Hence $\operatorname{K}(\mathbb{N}, 2)$ is not a quasi-category.
\end{proof}

The existence of a simplicial monoid whose underlying simplicial set is not a quasi-category is probably well known to experts. But the author could not find any written proof. 

Note that, for any monoid $M$, $\operatorname{K}(M, 1)$ is the nerve of $M$ which is understood as the single object 1-category. Thus by \cite[Proposition 1.8]{Joyal} it is a quasi-category. Moreover, since $\operatorname{K}(M, 0)$ is discrete, it is again a quasi-category. Thus $\operatorname{K}(\mathbb{N}, 2)$ is the lowest dimensional example of Eilenberg-MacLane spaces whose underlying simplicial sets are not quasi-categories.

Note also that Street defined at \cite[Example 1.3]{Street}, for any monoid (resp. commutative monoid) $M$ and each $n\in\{0, 1\}$ (resp. $n\in\mathbb{N}$), the Eilenberg-MacLane strict $\omega$-category $\operatorname{K}^{\omega}(M, n)$ which is actually a strict $n$-category. Thus our main theorem may not be surprising. In addition, since $\operatorname{K}(\mathbb{N}, 2)$ is not a quasi-category, \cite[Example 57]{Verity} does not show whether the functor $(-)^e$ maps it to a weak complicial set or not.
It might be possible to construct Eilenberg-MacLane spaces, in an appropriate sense, as higher categories directly, but the author is not familiar with such theories. At least we know that the $\omega$-nerve of $\operatorname{K}^{\omega}(M, n)$ is a (weak) complicial set by \cite[Theorem 266]{V2}.

\end{document}